\documentclass[12pt,twoside]{article}
\usepackage[left=1.5cm,right=1.5cm,top=3cm,bottom=2cm]{geometry}
\renewcommand{\baselinestretch}{1.0}
\usepackage{paralist}
\usepackage{marginnote}
\usepackage{tikz-cd}
\usepackage{amsfonts}
\usepackage{ntheorem}
\usepackage{amsrefs}
\usepackage{amssymb}
\usepackage{amsmath}
\usepackage{graphicx}
\usepackage{varwidth}
\usepackage{hyperref}
\usepackage{fancyhdr}
\usepackage{stmaryrd}
\usepackage{multirow}
\usepackage{blkarray}
\usepackage[autostyle]{csquotes}
\usepackage[mathscr]{euscript}
\hypersetup{
    colorlinks=true,
    linkcolor=black,
    filecolor=magenta,      
    urlcolor=red,
    citecolor=black
}
\DeclareMathAlphabet{\mathcalligra}{T1}{calligra}{c}{h}
\providecommand{\U}[1]{\protect\rule{.1in}{.1in}}
\newtheorem{theorem}{Theorem}

\newtheorem{proposition}[theorem]{Proposition}
\newtheorem{corollary}[theorem]{Corollary}
\newtheorem{lemma}[theorem]{Lemma}

\newtheorem*{problem}{Problem}

\newtheorem{remark}[theorem]{Remark}

\newcommand{\Ric}{\mathrm{Ric}}
\newcommand{\tr}{\mathrm{trace}}
\newcommand{\sca}{\mathsf{scal}}
\newcommand{\scal}{\mathrm{scal}}
\newcommand{\R}{\mathbb{R}}

\newcommand{\Z}{\mathbb{Z}}
\newcommand{\F}{\mathbb{F}}
\newcommand{\Vol}{\mathrm{Vol}}
\newcommand{\W}{\mathcal{W}}
\newcommand{\CU}{\mathcal{U}}

\newcommand{\Hom}{\mathrm{Hom}}
\newcommand\Sp{\mathbb{S}}
\newcommand\Sh{\mathcal{L}}

\newcommand\bea{\begin{eqnarray*}}
\newcommand\eea{\end{eqnarray*}}
\newcommand\be{\begin{equation}}
\newcommand\ee{\end{equation}}

\newcommand{\mb}{\mathbb}

\newcommand{\po}{{\hspace*{-1ex}}{\bf .  }}

\def\<{\langle}
\def\>{\rangle}
\newcommand\qed{\ifhmode\unskip\nobreak\fi\ifmmode\ifinner\else
\hskip5 pt \fi\fi\hbox{\hskip5 pt \vrule width4 pt height6 pt
depth1.5 pt \hskip 1pt }}

\begin{document}

\title{Topological obstructions for submanifolds in low codimension}
\author{Christos-Raent Onti and Theodoros Vlachos}
\date{}

\maketitle

\maketitle
\renewcommand{\thefootnote}{\fnsymbol{footnote}} 
\renewcommand{\thefootnote}{\arabic{footnote}} 

\renewcommand{\thefootnote}{\fnsymbol{footnote}} 
\renewcommand{\thefootnote}{\arabic{footnote}} 

\begin{abstract}
We prove integral curvature bounds in terms of the Betti numbers for 
compact submanifolds of the Euclidean space with low codimension. 
As an application, we obtain topological obstructions for $\delta$-pinched 
immersions. Furthermore, we obtain intrinsic obstructions
for minimal submanifolds in spheres with pinched second fundamental form.
\end{abstract}

\section{Introduction}

By the Nash's embedding theorem, every Riemannian manifold can
be isometrically immersed into a Euclidean space with sufficiently high 
codimension. On the other hand, there are results that impose restrictions on 
 isometric immersions with low codimension (cf. \cite{CK,Otsuki,JDM1,JDM2,JDM3,JDM4}). Most of 
these obstructions are pointwise conditions on the range of curvature. Here, we investigate 
obstructions for immersions with low codimension that involve total curvature. 
In particular, we are interested in the $L^{n/2}$-norm of the $(0,4)$-tensor 
$R-\big(\scal/n(n-1)\big)R_1$, where $R$ and $\scal$ denote the $(0,4)$-curvature 
tensor and the scalar curvature of the induced metric $g$ respectively, 
and $R_1=(1/2)g\varowedge g$, where $\varowedge$ 
stands for the Kulkarni-Nomizu product.
Shiohama and Xu \cite{SX} gave a lower bound in terms of the 
Betti numbers for compact hypersurfaces in the Euclidean space $\R^{n+1}$.
For higher codimension, they raised the following 

\begin{problem}\po
Let $M^n, n\geq 3,$ be a compact $n$-dimensional Riemannian manifold 
which admits an isometric immersion into $\R^{2n-1}$. Does there exist a constant 
$\varepsilon (n)$, depending only on $n$, such that if 
$$\int_{M^n}\Big\Vert R-\frac{\scal}{n(n-1)}R_1 \Big\Vert^{n/2}\  dM<\varepsilon(n)$$
then $M^n$ is homeomorphic to the sphere $\Sp^n$?
\end{problem}

In the present paper, we provide integral curvature bounds in terms of the 
Betti numbers, for compact submanifolds of Euclidean space with low 
codimension. As a consequence, we obtain partial answers 
to the above problem and extend previous ones given in \cite{Vlachos}. Throughout the paper, all manifolds under consideration 
are assumed to be without boundary, connected and oriented.
Our main result is stated as follows.

\begin{theorem}\po\label{MainTh2}
Given an integer $n\geq 4$ and $\delta\in(1/n,1)$, there exists a positive 
constant $c(n,\delta)$ such that if $M^n$ is a compact 
$n$-dimensional Riemannian manifold that admits an isometric immersion 
$f$ in $\R^{n+k},2\leq k\leq n/2,$ then 
$$
\int_{M^n}\Big\Vert R-\frac{\scal}{n(n-1)}R_1 \Big\Vert^{n/2}\  dM
+\int_{M^n} \big(S-\delta n^2H^2\big)_+^{n/2}\ dM
 \geq c(n,\delta)\sum_{i=k}^{n-k}\beta_i(M^n;\F),
$$
where $S$ is the squared norm of second fundamental form, 
$H$ the mean curvature\footnote{The mean curvature is given by $H=|\mathscr{H}|$, where $\mathscr{H}$ 
denotes the mean curvature vector.} of $f$, 
$\big(S-\delta n^2H^2\big)_+=\max\{S-\delta n^2H^2,0\}$ and $\beta_i(M^n;\F)$ the $i$-th 
Betti number of $M^n$ over an arbitrary coefficient field $\F$.
Furthermore,
\begin{enumerate}[(i)]
  \item If 
\begin{equation}\label{MIneq4}
\int_{M^n}\Big\Vert R-\frac{\scal}{n(n-1)}R_1 \Big\Vert^{n/2} \  dM+
\int_{M^n} \big(S-\delta n^2H^2\big)_+^{n/2}\ dM
< c(n,\delta),
\end{equation}
then $M^n$ has the homotopy type of a CW-complex with no cells of 
dimension $i$ for $k\leq i\leq n-k$. Moreover, if $k=2$, then the fundamental 
group $\pi_1(M^n)$ is a free group on $\beta_1(M^n;\Z)$ generators and if 
$\pi_1(M^n)$ is finite then $M^n$ is homeomorphic to $\Sp^n$.
  \item If the scalar curvature of $M^n$ is everywhere non-positive, then
  $$
  \int_{M^n}\Big\Vert R-\frac{\scal}{n(n-1)}R_1 \Big\Vert^{n/2}\  dM
  +\int_{M^n} \big(S-\delta n^2H^2\big)^{n/2}\ dM 
\geq c(n,\delta)\sum_{i=0}^{n}\beta_i(M^n;\F).
$$
  \item If the scalar curvature is everywhere non-positive and $$\int_{M^n}\Big\Vert R-\frac{\scal}{n(n-1)}R_1 \Big\Vert^{n/2}\  dM
  +\int_{M^n} \big(S-\delta n^2H^2\big)^{n/2}\ dM 
  < 3 c(n,\delta)$$ then 
  $M^n$ is homeomorphic to $\Sp^n$.
\end{enumerate}
\end{theorem}

In the case where (\ref{MIneq4}) is satisfied, the homology groups of $M^n$
must satisfy the condition $H_i(M^n;\F)=0$ for all $k\leq i\leq n-k$, where $\F$ is any 
coefficient field.

The idea of the proof is to relate the $L^{n/2}$-norm of the tensor 
$R-\big(\scal/n(n-1)\big)R_1$ with the Betti numbers using Morse theory, 
Chern-Lashof results \cite{CL1,CL2} and the Gauss equation. To this aim we
prove an algebraic inequality for symmetric bilinear forms (see Prop. \ref{CrProp2}). The presence of the 
integral $\int_{M^n} \big(S-\delta n^2H^2\big)_+^{n/2}\ dM$ in Theorem \ref{MainTh2} 
is essential since the algebraic inequality fails 
by dropping the corresponding term.

The above integral 
measures how far an immersion deviates from being {\it $\delta$-pinched}.
The latter means that the inequality $S\leq \delta n^2H^2$
holds everywhere, in which case $\delta\geq 1/n$.
We note that Shiohama and Xu \cite{ShXu1,ShXu2} gave a topological 
lower bound of the above integral in the case where $\delta=1/n$. 
The geometry and the topology of $\delta$-pinched immersions have been studied by several authors 
(see \cite{Knut,Andrews,BYChen,Cheng,ChengNona,Ar,Barbosa}) in the case where $\delta=1/(n-1)$.
Our results provide information on $\delta$-pinched immersions for any 
$\delta\in(1/n,1)$. Indeed, the following corollary follows immediately from Theorem \ref{MainTh2}
and gives an intrinsic obstruction to $\delta$-pinched immersions.

\begin{corollary}\po
If a compact $n$-dimensional Riemannian manifold $M^n, n\geq 4,$ 
admits an isometric $\delta$-pinched
immersion in $\R^{n+k},2\leq k\leq n/2,$ for some $\delta\in(1/n,1)$, then 
$$
\int_{M^n}\Big\Vert R-\frac{\scal}{n(n-1)}R_1 \Big\Vert^{n/2}\  dM
\geq c(n,\delta)\sum_{i=k}^{n-k}\beta_i(M^n;\F).
$$
In particular, if 
$$\int_{M^n}\Big\Vert R-\frac{\scal}{n(n-1)}R_1 \Big\Vert^{n/2}\  dM
< c(n,\delta),$$ then $M^n$ has the homotopy type of a 
CW-complex with no cells of dimension $i$ for $k\leq i\leq n-k$. 
Moreover, if $k=2$, then $\pi_1(M^n)$ is a free group on $\beta_1(M^n;\Z)$ generators and if 
$\pi_1(M^n)$ is finite then $M^n$ is homeomorphic to $\Sp^n$.
\end{corollary}

The next results are easy consequences of Theorem \ref{MainTh2}
and provide partial answers to the problem raised by Shiohama and Xu.

\begin{corollary}\po\label{pinched3}
If a compact $n$-dimensional Riemannian manifold 
$M^n$, $n\geq 4,$ admits an isometric immersion $f$ in 
$\R^{n+k},2\leq k\leq n/2,$ such that 
$$
\int_{M^n}\Big\Vert R-\frac{\scal}{n(n-1)}R_1 \Big\Vert^{n/2}\  dM
< \lambda\  c(n,\delta)
$$
and
$$
\int_{M^n} \big(S-\delta n^2H^2\big)_+^{n/2}\ 
dM\leq (1-\lambda)\  c(n,\delta)\sum_{i=k}^{n-k}\beta_i(M^n;\F),
$$
where $\lambda\in(0,1)$ and $\delta\in(1/n,1)$, then $f$ is $\delta$-pinched 
and $M^n$ has the homotopy type of a CW-complex with no cells of 
dimension $i$ for $k\leq i\leq n-k$. 
Furthermore,
\begin{enumerate}[(i)]
  \item If $k=2$, then $\pi_1(M^n)$ is a free group on 
$\beta_1(M^n;\Z)$ generators and if 
$\pi_1(M^n)$ is finite then $M^n$ is homeomorphic to $\Sp^n$.
  \item If the mean curvature is everywhere positive and 
$\delta=1/(n-1)$, then $M^n$ is diffeomorphic to $\Sp^{n}$. 
\end{enumerate}
\end{corollary}

\begin{corollary}\po\label{Cor1}
If a compact $n$-dimensional Riemannian manifold 
$M^n$, $n\geq 4,$ admits an isometric immersion in 
$\R^{n+k},2\leq k\leq n/2,$ such that
$$
\int_{M^n}\Big\Vert R-\frac{\scal}{n(n-1)}R_1 \Big\Vert^{n/2}\  dM
\leq \lambda\  c(n,\delta)\sum_{i=k}^{n-k}\beta_i(M^n;\F)
$$
and
$$
\int_{M^n} \big(S-\delta n^2H^2\big)_+^{n/2}\ dM
< (1-\lambda)\  c(n,\delta),
$$
where $\lambda\in(0,1)$ and $\delta\in(1/n,1)$, then $M^n$ is isometric to a constant 
curvature sphere.
\end{corollary}

Minimal submanifolds with pinched second fundamental form
have been studied by Simons \cite{simons}, Chern, do Carmo, 
Kobayashi \cite{ChDoKo} and Leung \cite{Leung}, among others.
We provide intrinsic obstructions for minimal submanifolds in spheres
with sufficiently pinched second fundamental form.

\begin{corollary}\po\label{ThmMin2}
Let $f\colon M^n\rightarrow \Sp^{n+k-1}, 2\leq k\leq n/2,$ be an 
isometric minimal immersion of a compact $n$-dimensional Riemannian 
manifold $M^n, n\geq 4$. If the squared norm of the second fundamental form 
satisfies $S\leq n(\delta n-1)$ for some $\delta\in(1/n,1)$, then 
$$
\int_{M^n}\Big\Vert R-\frac{\scal}{n(n-1)}R_1 \Big\Vert^{n/2}\ dM
\geq c(n,\delta)\sum_{i=k}^{n-k}\beta_i(M^n;\F).
$$
In particular, if $$\int_{M^n}\Big\Vert R-\frac{\scal}{n(n-1)}R_1 \Big\Vert^{n/2}\ dM 
< c(n,\delta),$$ then $M^n$ has the homotopy type of a CW-complex with no cells 
of dimension $i$ for $k\leq i\leq n-k$.  Moreover, if $k=2$, then the fundamental 
group $\pi_1(M^n)$ is a free group on $\beta_1(M^n;\Z)$ generators and if 
$\pi_1(M^n)$ is finite then $M^n$ is homeomorphic to $\Sp^n$.
\end{corollary}

It is well known that the Weyl tensor $\W$ of a $n$-dimensional 
Riemannian manifold $M^n, n\geq 4,$ vanishes if and only if $M^n$
is conformally flat.
The study of conformally flat manifolds, from the point of view 
of submanifold theory, was initiated by Cartan 
in \cite{ECartan}. The case of compact conformally flat hypersurfaces of 
Euclidean space has been studied by Do Carmo, Dajczer and Mercuri \cite{MMF}.
For low codimension $k$, Moore \cite{JDM2} proved that such submanifolds 
have the homotopy type of a 
CW-complex with no cells of dimension $i,\ k<i<n-k$.
Therefore, it is natural to seek for restrictions on the 
topology of compact almost conformally flat submanifolds, in the sense 
that the Weyl tensor is sufficiently small in a suitable norm.

The case of hypersurfaces has been
treated in \cite{OV}. In this paper, we prove an inequality for the 
$L^{n/2}$-norm of the Weyl tensor for compact $n$-dimensional 
Riemannian manifolds that allow conformal immersions
in the Euclidean space with low codimension. As a consequence,
 we obtain a partial answer to the above question.

\begin{theorem}\po\label{MainTh}
Given $n\geq 6$ and $\delta\in(1/n,1)$, there exists a positive 
constant $c_1(n,\delta)$ such that if $M^n$ is a compact 
$n$-dimensional Riemannian manifold that admits a conformal 
immersion in $\R^{n+k},2\leq k\leq [(n-2)/2],$ then 
$$
\int_{M^n}\Vert \mathcal{W} \Vert^{n/2}\ dM
+\int_{M^n} \big(S-\delta n^2H^2\big)_+^{n/2}\ dM
 \geq c_1(n,\delta)\sum_{i=k+1}^{n-k-1}\beta_i(M^n;\F).
$$
In particular, if 
$$
\int_{M^n}\Vert \mathcal{W} \Vert^{n/2}\ dM+
\int_{M^n} \big(S-\delta n^2H^2\big)_+^{n/2}\ dM
< c_1(n,\delta),
$$
then $M^n$ has the homotopy type of a CW-complex with no cells of 
dimension $i$ for $k<i<n-k$.
\end{theorem}

As an application of Theorem \ref{MainTh}, we may obtain results similar to
Corollaries 2-5 for the Weyl tensor instead of the tensor $R-\big(\scal/n(n-1)\big)R_1$. For instance,
we have the following

\begin{corollary}\po
If a compact $n$-dimensional Riemannian manifold $M^n, 
n\geq 6,$ admits a conformal $\delta$-pinched
immersion in $\R^{n+k},2\leq k
\leq [(n-2)/2],$ for some $\delta\in(1/n,1)$, then 
$$
\int_{M^n}\Vert \mathcal{W} \Vert^{n/2}\ dM
\geq c_1(n,\delta)\sum_{i=k+1}^{n-k-1}\beta_i(M^n;\F).
$$
In particular, if 
$$\int_{M^n}\Vert \mathcal{W} \Vert^{n/2}\ dM< c_1(n,\delta),$$ 
then $M^n$ has the homotopy type of a 
CW-complex with no cells of dimension $i$ for $k<i<n-k$. 
\end{corollary}

\section{Algebraic auxiliary results}

This section is devoted to some algebraic results that are crucial for the proofs. 
Let $V$ and $W$ be finite dimensional real vector spaces equipped with non-degenerate inner 
products which, by abuse of notation, are both denoted by $\langle\cdot,\cdot\rangle$.
The inner product of $V$ is assumed to be positive definite. 
We denote by $\Hom(V\times V,W)$ the space 
of all bilinear forms and by $\mathrm{Sym}(V\times V,W)$ its subspace that consists 
of all symmetric bilinear forms. The space $\mathrm{Sym}(V\times V,W)$ can be viewed as 
a complete metric space with respect to the usual Euclidean norm $\Vert\cdot\Vert$.

The {\it Kulkarni-Nomizu product} of two bilinear forms 
$\phi,\psi\in \mathrm{Hom}(V\times V,\mb{R})$ 
is the $(0,4)$-tensor 
$\phi\varowedge \psi\colon V\times V\times V\times V \rightarrow \mb{R}$ 
defined by 
\bea
\phi\varowedge \psi(x_1,x_2,x_3,x_4) &=& 
\phi(x_1,x_3) \psi(x_2,x_4)+\phi(x_2,x_4) \psi(x_1,x_3) \\
 & &-\phi(x_1,x_4) \psi(x_2,x_3)-\phi(x_2,x_3)\psi(x_1,x_4).
\eea

Using the inner product of $W$, we 
extend the {\it Kulkarni-Nomizu product} to bilinear forms 
$\beta,\gamma\in\Hom(V\times V,W)$, as the $(0,4)$-tensor 
$\beta\varowedge \gamma\colon V\times V\times V\times V \rightarrow \mb{R}$ 
defined by 
\bea
\beta\varowedge\gamma(x_1,x_2,x_3,x_4) \!\!\!&=&\!\!\! 
\langle\beta(x_1,x_3),\gamma(x_2,x_4) \rangle 
-\langle\beta(x_1,x_4),\gamma(x_2,x_3) \rangle \\
\!\!\!& &\!\!\! +\langle\beta(x_2,x_4),\gamma(x_1,x_3) \rangle-
\langle\beta(x_2,x_3),\gamma(x_1,x_4) \rangle.
\eea

A bilinear form $\beta\in \mathrm{Hom}(V\times V,W)$ is called {\it{flat}} with respect 
to the inner product of $W$ if
$$\langle\beta(x_1,x_3),\beta(x_2,x_4)\rangle
-\langle\beta(x_1,x_4),\beta(x_2,x_3)\rangle=0$$ 
for all $x_1,x_2,x_3,x_4\in V$, or equivalently if $\beta\varowedge\beta=0.$

Associated to each bilinear form $\beta$ is the {\it nullity space} $\mathcal{N}(\beta)$ 
defined by $$\mathcal{N}(\beta)=\{x\in V\ :\ \beta(x,y)=0\ \ \text{for all}\ \ y\in V\}.$$

We need the following lemma, which was given in \cite[Lemma 2.1]{Vlachos}.

\begin{lemma}\po\label{ThVLe}
 Let $\beta\in\mathrm{Sym}(V\times V,W)$ be a bilinear form, where $V$ and $W$ are both equipped with positive definite 
 inner products and $\dim W\leq \dim V-2$. If $\beta \varowedge \beta= \mu \ \langle\cdot,\cdot\rangle\varowedge \langle\cdot,\cdot\rangle$ for some $\mu\neq 0$, then $\mu>0$ and there exist a unit vector $\xi\in W$ and a subspace $V_1\subseteq V$ such that 
$$\dim V_1\geq \dim V-\dim W+1$$ and
$$\beta(x,y)=\sqrt{\mu} \  \langle x,y\rangle\xi,\ \text{for all}\ x\in V_1\ \text{and}\ y\in V.$$
 \end{lemma}

\smallskip

We define the map 
${\sf scal}\colon \mathrm{Sym}(V\times V,W)\rightarrow \R$ 
by 
$${\sf scal}(\beta)=\tr \  {\sf Ric} (\beta),$$ where 
$${\sf Ric}(\beta)(x,y)=\tr \  {\sf R}(\beta)(\cdot,x,\cdot,y),\ \ x,y\in V\ \ \text{and}\ \ {\sf R}(\beta) = \frac{1}{2}\beta\varowedge\beta.$$

Hereafter, we assume that $V$ and $W$ are both endowed with positive 
definite inner products. For each $\beta\in \mathrm{Sym}(V\times V,W)$, 
we define the map
$$
\beta^\sharp\colon W\rightarrow \mathrm{End}(V),\ \ \xi\mapsto \beta^\sharp(\xi)
$$
such that
$$
\langle \beta^\sharp(x),y\rangle=\langle \beta(x,y),\xi\rangle,\ \ 
\text{for all}\ x,y\in V,
$$
where $\mathrm{End}(V)$ denotes the set of all selfadjoint endomorphisms 
of $V$.

\smallskip

Let $\dim V=n$ and $\dim W=k$. When $2\leq k\leq n/2$, for each 
$\beta\in \mathrm{Sym}(V\times V,W)$, we denote by $\Phi(\beta)$ the  
subset of the unit $(k-1)$-sphere $\Sp^{k-1}$ in $W$ given by
 $$\Phi(\beta)=\{u\in\mathbb{S}^{k-1}\ :\ k\leq\mathrm{Index}\ \beta^\sharp(u)\leq n-k\}.$$

The following proposition is crucial for the proof of Theorem \ref{MainTh2}.

\begin{proposition}\po\label{CrProp2}
Given integers $2\leq k\leq n/2$ and $\lambda\in(1/n,1)$, there exists a positive constant $\varepsilon(n,k,\lambda)>0$, such that the following inequality holds
\begin{equation}\label{CrIneq2}
\frac{1}{4}\Big\Vert\beta\varowedge\beta-\frac{\sca(\beta)}{n(n-1)}\langle\cdot,\cdot\rangle\varowedge \langle\cdot,\cdot\rangle\Big\Vert^2
+\big(\Vert\beta\Vert^2-\lambda\vert\tr \  {\beta}\vert^2\big)_+^2
\geq \varepsilon(n,k,\lambda)\Big(\int_{\Lambda(\beta)}\vert\det\beta^\sharp (u)\vert d\Sp_u\Big)^{4/n}
\end{equation} for any $\beta\in \mathrm{Sym}(V\times V, W)$, where 
\begin{equation}\nonumber
\Lambda(\beta)=
\begin{cases}
\Phi(\beta), &\ \text{if}\ \ \sca(\beta)>0 \\
\Sp^{k-1}, &\ \text{if}\ \ \sca(\beta)\leq 0.
\end{cases}
\end{equation}
\end{proposition}

\begin{proof}
We consider the functions $\phi_\lambda,\psi\colon \mathrm{Sym}(V\times V,W)\rightarrow \mathbb{R}$ 
defined by
$$\phi_\lambda(\beta)=\frac{1}{4}\Big\Vert\beta\varowedge\beta-
\frac{\sca(\beta)}{n(n-1)}\langle\cdot,\cdot\rangle\varowedge \langle\cdot,\cdot\rangle\Big\Vert^2
+\big(\Vert\beta\Vert^2-\lambda\vert\tr \  {\beta}\vert^2\big)_+^2$$
and $$\psi(\beta)=\int_{\Lambda(\beta)}|\det\beta^\sharp(u)|\ d\Sp_u.$$
We shall prove that $\phi_\lambda$ attains a positive minimum on $\Sigma_{n,k}$, where
$$\Sigma_{n,k}=\{\beta\in\mathrm{Sym}(V\times V,W)\ :\ \psi(\beta)=1\}.$$
There exists a sequence $\{\beta_m\}$ in $\Sigma_{n,k}$ such that 
$$\lim_{m\rightarrow\infty}\phi_\lambda(\beta_m)=\inf \phi_\lambda(\Sigma_{n,k})\geq 0.$$
We observe that $\beta_m\neq 0$ for all $m\in \mathbb{N}$, since $\beta_m\in\Sigma_{n,k}$. Then 
we may write $\beta_m=\Vert\beta_m\Vert\widehat{\beta}_m,$ where $\Vert\widehat{\beta}_m\Vert=1$.

We claim that the sequence $\{\beta_m\}$ is bounded. Assume to the contrary that there exists a 
subsequence of $\{\beta_m\}$, which by abuse of notation is again denoted by $\{\beta_m\}$, such that 
$\lim_{m\rightarrow \infty}\Vert\beta_m\Vert=\infty.$
Since $\Vert\widehat{\beta}_m\Vert=1$, we may assume, by taking a subsequence if necessary, that 
$\{\widehat{\beta}_m\}$ converges to some $\widehat{\beta}\in\mathrm{Sym}(V\times V,W)$ with 
$\Vert\widehat{\beta}\Vert=1$. Using the fact that $\phi_\lambda$ is homogeneous of degree $4$, 
we have $\phi_\lambda(\widehat{\beta}_m)=\phi_\lambda(\beta_m)/\Vert\beta_m\Vert^4.$
Thus $\lim_{m\rightarrow \infty}\phi_\lambda(\widehat{\beta}_m)=0$ and consequently 
$\phi_{\lambda}(\widehat{\beta})=0$, or equivalently 
\begin{equation}\nonumber
\widehat{\beta}\varowedge\widehat{\beta}=
\frac{\sca(\widehat{\beta})}{n(n-1)} \langle\cdot,\cdot\rangle\varowedge \langle\cdot,\cdot\rangle
\end{equation}
and
\begin{equation}\label{Eq11}
1\leq \lambda\vert\tr \ \ \widehat{\beta}\vert^2=\lambda(\sca(\widehat{\beta})+1).
\end{equation}
Since $\lambda<1$, equation (\ref{Eq11}) implies $\sca(\widehat{\beta})>0$. According to Lemma \ref{ThVLe}, 
there exists a unit vector $\widehat{\xi}\in W$ and subspace $\widehat{V}_1$ of $V$ with $\dim \widehat{V}_1\geq n-k+1$ such that 
\begin{equation}\label{eq22}
\widehat{\beta}(x,y)=\Big(\frac{\sca(\widehat{\beta})}{n(n-1)}\Big)^{1/2} \langle x,y\rangle\widehat{\xi}
\ \ \text{for all} \ \ x\in \widehat{V}_1\ \ \text{and}\ \ y\in V.
\end{equation}
Moreover, since $\{\beta_m\}$ is in $\Sigma_{n,k}$, there exists an open subset $\widehat{\CU}_m$
of $\Sp^{k-1}$ such that $\widehat{\CU}_m\subseteq \Lambda(\widehat{\beta}_m)$ and
$\det\widehat{\beta}_m^\sharp(u)\neq0$ for all  $u\in \widehat{\CU}_m$ and $m\in \mathbb{N}.$
From $\sca(\widehat{\beta})>0$, we deduce that $\sca(\widehat{\beta}_m)>0$ and so 
$\widehat{\CU}_m\subseteq \Phi(\widehat{\beta}_m)$ for $m$ large enough.

Let $\{\widehat{u}_m\}$ be a sequence such that $\widehat{u}_m\in \widehat{\CU}_m$ for all $m\in\mathbb{N}$. 
We may assume that $\{\widehat{u}_m\}$ is convergent, by passing if 
necessary to a subsequence and set $\widehat{u}=\lim_{m\rightarrow \infty}\widehat{u}_m$. Since $\lim_{m\rightarrow \infty}\widehat{\beta}_m^\sharp(\widehat{u}_m)=\widehat{\beta}^\sharp(\widehat{u})$ and $\widehat{u}_m\in\widehat{\CU}_m,$ we deduce 
that $\mathrm{Index}\ \widehat{\beta}^\sharp(\widehat{u}) \leq n-k.$ Then, from (\ref{eq22}) we obtain
$\langle \widehat{\xi},\widehat{u}\rangle\geq 0$. We claim that $\langle \widehat{\xi},\widehat{u}\rangle= 0$.
Indeed, if $\langle\widehat{\xi},\widehat{u}\rangle>0$, then (\ref{eq22}) implies that $\widehat{\beta}^\sharp(\widehat{u})$ has 
at least $n-k+1$ positive eigenvalues and so, for $m$ large enough, $\widehat{\beta}_m^\sharp(\widehat{u}_m)$ has at least 
$n-k+1$ positive eigenvalues. This and the fact that $\det\widehat{\beta}_m^\sharp(u)\neq 0$ for all $u\in \widehat{\CU}_m$, shows
that $\widehat{\beta}_m^\sharp(\widehat{u}_m)$ has at most $k-1$ negative eigenvalues. Therefore, 
$\mathrm{Index}\ \widehat{\beta}_m^\sharp(\widehat{u}_m)\leq k-1$ which is a contradiction, since $\widehat{u}_m\in \widehat{\CU}_m$. 

Thus, we have proved that for any convergent sequence $\{\widehat{u}_m\}$ such that $\widehat{u}_m\in \widehat{\CU}_m$ for all $m$, we have 
$\langle\lim_{m\rightarrow \infty}\widehat{u}_m,\widehat{\xi}\rangle=0.$

Since $\widehat{\CU}_m$ is open, we may choose convergent sequences  
$\{\widehat{u}_m^{(1)}\},\{\widehat{u}_m^{(2)}\},\cdots,\{\widehat{u}_m^{(k)}\}$ in $\widehat{\CU}_m$ such that 
$\widehat{u}_m^{(1)}, \widehat{u}_m^{(2)},\cdots, \widehat{u}_m^{(k)}$ span $W$ 
for all $m\in \mathbb{N}.$ From (\ref{eq22}) and the fact that $\langle\lim_{m\rightarrow \infty}\widehat{u}_m^{(a)},\widehat{\xi}\rangle=0$
for all $a\in\{1,2,\cdots,k\}$, we obtain that the restriction of $\widehat{\beta}_m$ to $\widehat{V}_1\times\widehat{V}_1$ 
satisfies
$$\lim_{m\rightarrow\infty} \widehat{\beta}_m\vert_{\widehat{V}_1\times\widehat{V}_1}=0$$ and consequently
\be\label{xx1}
\lim_{m\rightarrow\infty} (\widehat{\beta}_m\varowedge\widehat{\beta}_m)\vert_{\widehat{V}_1\times\widehat{V}_1\times
\widehat{V}_1\times\widehat{V}_1}=0.
\ee
From the inequality 
$$
\frac{1}{4}\Big\Vert\Big(\widehat{\beta}_m\varowedge\widehat{\beta}_m-
\frac{\sca(\widehat{\beta}_m)}{n(n-1)}\langle\cdot,\cdot\rangle\varowedge \langle\cdot,\cdot\rangle\Big)
\big\vert_{\widehat{V}_1\times\widehat{V}_1\times
\widehat{V}_1\times\widehat{V}_1}\Big\Vert^2
\leq\phi_\lambda(\widehat{\beta}_m),
$$
(\ref{xx1}) and the fact that $\lim_{m\rightarrow \infty}\phi_\lambda(\widehat{\beta}_m)=0$ we obtain $\sca(\widehat{\beta})=0$,
which contradicts (\ref{Eq11}). 
Thus, the sequence $\{\beta_m\}$ is bounded, and it converges to some $\beta\in\mathrm{Sym}(V\times V,W)$, by taking a
subsequence if necessary.

We claim that $\phi_\lambda(\beta)>0$. Arguing indirectly, we assume that $\phi_\lambda(\beta)=0$. Then 
\begin{equation}\nonumber
\beta\varowedge\beta=\frac{\sca(\beta)}{n(n-1)} \langle\cdot,\cdot\rangle\varowedge \langle\cdot,\cdot\rangle
\end{equation} 
and
\begin{equation}\label{Eq111}
\Vert\beta\Vert^2\leq \lambda \vert\tr \  \beta\vert^2=\lambda(\sca(\beta)+\Vert\beta\Vert^2).
\end{equation}
We notice that $\beta\neq0$. Indeed, if $\beta=0$, then $\beta^\sharp(u)=0$ for all $u\in \Sp^{k-1}$. 
Since $\beta_m\in \Sigma_{n,k}$ for all $m\in \mb{N}$, there exists $\xi_m\in \Lambda(\beta_m)$ such that 
\begin{equation}\label{3777}
|\det\beta_m^\sharp(\xi_m)|\Vol(\Lambda(\beta_m))=1\ \ \text{for all} \ m\in\mathbb{N}.
\end{equation}
We may assume that $\xi_m$ converges to some $\xi$, 
 by passing to a subsequence if necessary. Then $\lim_{m\rightarrow \infty}\beta_m^\sharp(\xi_m)=\beta^\sharp(\xi)=0,$ which 
 contradicts (\ref{3777}). Therefore $\beta\neq 0$.

Now, from (\ref{Eq111}) we obtain that $\sca(\beta)\neq 0$. Then, Lemma \ref{ThVLe} implies that $\sca(\beta)>0$ and 
there exists a unit vector $\xi\in W$ and a subspace $V_1$ of $V$ with $\dim V_1\geq n-k+1$ such that
\begin{equation}\label{eq266}
\beta(x,y)=\Big(\frac{\sca(\beta)}{n(n-1)}\Big)^{1/2} \langle x,y\rangle\xi,\ \ \text{for all}\  x\in V_1\ \ \text{and}\ y\in V.
\end{equation}

Since $\beta_m\in \Sigma_{n,k}$ for all $m\in \mathbb{N}$, there exists an open subset $\CU_m$ of $\Sp^{k-1}$ such that
$\CU_m\subseteq \Lambda(\beta_m)$ and $\det\beta_m^\sharp(u)\neq0$, for all $u\in \CU_m$ and $m\in \mathbb{N}.$
Moreover, we have that $\sca(\beta_m)>0$ and so 
$\CU_m\subseteq \Phi(\beta_m)$ for $m$ large enough.

Let $\{u_m\}$ be a sequence such that $u_m\in \CU_m$ for all $m\in\mathbb{N}$. We 
may assume that $u_m$ is convergent, by passing if necessary to a subsequence and set $u=\lim_{m\rightarrow \infty}u_m.$ Since 
$\lim_{m\rightarrow \infty}\beta_m^\sharp(u_m)=\beta^\sharp(u)$ and $u_m\in\CU_m$ it follows that $\mathrm{Index}\ \beta^\sharp(u) \leq n-k$.
Then, from (\ref{eq266}) we get $\langle \xi,u\rangle\geq 0$. We claim that $\langle\xi,u\rangle=0$. Indeed, if $\langle\xi,u\rangle>0$ then (\ref{eq266}) implies that $\beta^\sharp(u)$ has at least $n-k+1$ positive eigenvalues and so, for $m$ large enough, $\beta_m^\sharp(u_m)$ has at least $n-k+1$ positive eigenvalues. This and the fact that $\det \beta_m^\sharp(u)\neq 0$ for all $u\in \CU_m$ and $m\in \mathbb{N},$
shows that $\beta_m^\sharp(u_m)$ has at most $k-1$ negative eigenvalues, which is a contradiction, since $u_m\in \CU_m$ for all $m\in \mathbb{N}$. 
  
  Thus, we have proved that for any convergent sequence $\{u_m\}$ such that $u_m\in \CU_m$ for all $m$, we have 
$\langle\lim_{m\rightarrow \infty}u_m,\xi\rangle=0.$
Since $\CU_m$ is open, we may choose convergent sequences  $\{u_m^{(1)}\},\{u_m^{(2)}\},\cdots,\{u_m^{(k)}\}$ 
in $\CU_m$ such that $u_m^{(1)}, u_m^{(2)},\cdots, u_m^{(k)}$ span $W$ for all $m\in \mathbb{N}.$ Then, from (\ref{eq266}) 
and the fact that $\langle\lim_{m\rightarrow \infty}u_m^{(a)},\xi\rangle=0$ for all $a\in\{1,2,\cdots,k\}$ we obtain that the restriction
of $\beta_m$ to $V_1\times V_1$ satisfies
$$\lim_{m\rightarrow\infty}\beta_m\vert_{V_1\times V_1}=0$$ and consequently
\be\label{xx2}
\lim_{m\rightarrow\infty} (\beta_m\varowedge\beta_m)\vert_{V_1\times V_1\times V_1\times V_1}=0.
\ee
From the inequality 
$$
\frac{1}{4}\Big\Vert\Big(\beta_m\varowedge \beta_m-
\frac{\sca(\beta_m)}{n(n-1)}\langle\cdot,\cdot\rangle\varowedge \langle\cdot,\cdot\rangle\Big)
\big\vert_{V_1\times V_1\times V_1\times V_1}\Big\Vert^2
\leq\phi_\lambda(\beta_m),
$$
(\ref{xx2}) and the fact that $\lim_{m\rightarrow \infty}\phi_\lambda(\beta_m)=0$ we obtain $\sca(\beta)=0$,
which contradicts (\ref{Eq111}). 
Thus, we have proved that $\phi_\lambda(\beta)>0$ and so $\phi_\lambda$ attains a positive minimum on $\Sigma_{n,k}$ which obviously depends only on $n,k$ and $\lambda$ and is denoted by $\varepsilon(n,k,\lambda)$. 

Now, let $\beta\in \mathrm{Sym}(V\times V, W)$. Assume that $\psi(\beta)\neq 0$ and set $\gamma=\beta/(\psi(\beta))^{1/n}.$ Clearly 
$\gamma\in\Sigma_{n,k}$, and consequently $\phi_\lambda(\gamma)\geq \varepsilon(n,k,\lambda).$ Since $\phi_\lambda$ is 
homogeneous of degree $4$, the desired inequality is obviously fulfilled. In the case where $\psi(\beta)=0$, the inequality is trivial. \qed
\end{proof}

\bigskip

We also need the following result on flat bilinear forms, which is due to Moore 
\cite[Proposition 2]{JDM2}.

\begin{lemma}\po\label{JDMLe}
Let $\beta\in\mathrm{Sym}(V\times V,U)$ be a flat bilinear form with respect to
 a Lorentzian inner product of $U$. If $\dim V>\dim U$ and $\beta(x,x)\neq 0$ 
 for all non-zero $x\in V$, then there is a non-zero isotropic vector $e\in U$ and 
 a bilinear form $\phi\in \mathrm{Sym}(V\times V, \R)$ such that 
 $\dim \mathcal{N}(\beta-e\phi)\geq \dim V-\dim U+2.$
\end{lemma}

We define the map 
${\sf W}\colon \mathrm{Sym}(V\times V,W)\rightarrow\mathrm{Hom}(V\times V\times V\times V,\R)$ 
by 
$${\sf W}(\beta) = {\sf R}(\beta)-{\sf L}(\beta)\varowedge \langle\cdot,\cdot\rangle,$$ where 
$${\sf L}(\beta)=\frac{1}{n-2}\Big({\sf Ric}(\beta)-\frac{{\sf scal}(\beta)}{2(n-1)}\langle\cdot,\cdot\rangle\Big),$$

The following lemma is in fact contained in \cite{JDM2}. For the sake of completeness we give
a short proof.
 
\begin{lemma}\po\label{JDMLe2}
Let $\beta\in \mathrm{Sym}(V\times V, W)$ be a bilinear form
 and $\dim W< \dim V-2$. If ${\sf W}(\beta)=0$, 
then there exists a vector $\xi\in W$ 
and a subspace $V_1\subseteq V$ such that 
$$\dim V_1\geq \dim V-\dim W$$
and $$\beta(x,y)=\langle x,y\rangle \xi,\ \ 
\text{for all}\ x\in V_1\ \text{and}\ y\in V.$$
\end{lemma}

\begin{proof}
We endow the vector space 
$\widetilde{W}=W\oplus \R^2$ with the Lorentzian inner product 
$\langle\langle\cdot,\cdot\rangle\rangle$ given by
$$\langle\langle\ (\xi,(s_1,s_2)), (\eta,(t_1,t_2))\  \rangle\rangle=
\langle \xi,\eta\rangle+s_1t_2+s_2t_1$$ and define the symmetric 
bilinear form $\widetilde{\beta}\colon V\times V\rightarrow \widetilde{W}$ 
by $$\widetilde{\beta}(x,y)=\big(\beta(x,y),\langle x,y \rangle, -{\sf L}(\beta)(x,y)\big).$$ 
Since $\sf{W}(\beta)=0$ it follows that $\widetilde{\beta}$ is flat with 
respect to $\langle\langle\cdot , \cdot \rangle\rangle$. From Lemma 
\ref{JDMLe}, we know that there exists a non-zero isotropic vector 
$e=(\eta,(s,t))\in\widetilde{W}$ and a symmetric bilinear form 
$\phi:V\times V\rightarrow \mathbb{R}$ such that 
$\dim \mathcal{N}(\widetilde\beta-\phi e)\geq \dim V-\dim W$. 
By setting $V_1=\mathcal{N}(\widetilde{\beta}-e\phi)$, we have that 
$\widetilde{\beta}(x,y)=\phi(x,y) e,$ or equivalently 
$$\beta(x,y) = \phi(x,y)\eta, \ \ \langle x,y \rangle = s\phi(x,y) \ \ 
\text{and}\ \ {\sf L}(\beta)(x,y)=-t\phi(x,y),$$ 
for all $x\in V_1\ \text{and}\ y\in V.$ Therefore, 
$\beta(x,y) = \langle x,y \rangle \xi$,
where $\xi=(1/s)\eta$. \qed
\end{proof}

\medskip

When $2\leq k\leq [(n-2)/2]$, then for each $\beta\in \mathrm{Sym}(V\times V,W)$, 
we define the subset $\Omega(\beta)$ of the unit $(k-1)$-sphere 
$\Sp^{k-1}$ in $W$ given by
$$\Omega(\beta)=\{u\in\mathbb{S}^{k-1}\ :\ k<\mathrm{Index}\ \beta^\sharp(u)<n-k\}.$$

The following proposition is crucial for the proof of Theorem \ref{MainTh}.

\begin{proposition}\po\label{CrProp}
Given integers $2\leq k\leq [(n-2)/2]$ and $ \lambda\in(1/n,1)$, there exists a 
positive constant $\varepsilon_1(n,k,\lambda)$ such that the following inequality holds 
\begin{equation}\label{CrIneq}
\Vert \mathsf{W}(\beta)\Vert^2+\big(\Vert\beta\Vert^2-\lambda\vert\tr \ {\beta}\vert^2\big)_+^2
\geq \varepsilon_1(n,k,\lambda)\Big(\int_{\Omega(\beta)}|\det\beta^\sharp(u)|\ d\Sp_u\Big)^{4/n},
\end{equation}
 for all $\beta\in \mathrm{Sym}(V\times V,W)$.
 \end{proposition}

\begin{proof}
We consider the functions 
$\phi_\lambda,\psi\colon \mathrm{Sym}(V\times V,W)\rightarrow \mathbb{R}$ defined by
$$\phi_\lambda(\beta) = \Vert \mathsf{W}(\beta)\Vert^2+
\big(\Vert\beta\Vert^2-\lambda\vert\tr \ {\beta}\vert^2\big)_+^2\ \ \ 
\text{and}\ \ \ \psi(\beta) = \int_{\Omega(\beta)}|\det\beta^\sharp(u)|\ d\Sp_u.$$
In order to prove the desired inequality, it is sufficient to show that $\phi_\lambda$ 
attaints a positive minimum on $\Sigma_{n,k}$, where
$$\Sigma_{n,k}=\{\beta\in\mathrm{Sym}(V\times V,W)\ :\ \psi(\beta)=1\}.$$\
There exists a sequence $\{\beta_m\}$ in $\Sigma_{n,k}$ such that 
$$\lim_{m\rightarrow\infty}\phi_\lambda(\beta_m)=\inf \phi_\lambda(\Sigma_{n,k})\geq 0.$$ 
We observe that $\beta_m\neq0$ for all $m\in\mb{N}$, since $\beta_m\in \Sigma_{n,k}$. 
Thus, we may write $\beta_m=\Vert\beta_m\Vert\widehat{\beta}_m,$ where 
$\Vert\widehat{\beta}_m\Vert=1$.

We claim that the sequence $\{\beta_m\}$ is bounded. Assume to the contrary that 
there exists a subsequence of $\{\beta_m\}$, which by abuse of notation is again 
denoted by $\{\beta_m\}$, such that $\lim_{m\rightarrow \infty}\Vert\beta_m\Vert=\infty.$ 
Since $\Vert\widehat{\beta}_m\Vert=1$, we may assume, by taking a subsequence if 
necessary, that $\{\widehat{\beta}_m\}$ converges to some $\widehat{\beta}\in\mathrm{Sym}(V\times V,W)$ 
with $\Vert\widehat{\beta}\Vert=1$. Using the fact that $\phi_\lambda$ is homogeneous of degree $4$, 
we obtain $\phi_\lambda(\widehat{\beta}_m)=\phi_\lambda(\beta_m)/\Vert\beta_m\Vert^4.$
Thus $\lim_{m\rightarrow\infty}\phi_\lambda(\widehat{\beta}_m)=0$ and consequently 
$\phi_\lambda(\widehat{\beta})=0$, or equivalently $\mathsf{W}(\widehat{\beta})=0$ and
\begin{equation}\label{CrEq2}
1\leq \lambda \vert\tr \  {\widehat\beta}\vert^2=\lambda(\sca(\widehat\beta)+1).
\end{equation}
According to Lemma \ref{JDMLe2}, we have that there exist a vector subspace $\widehat{V}_1$ 
of $V$ with $\dim \widehat{V}_1\geq n-k$ and a vector $\widehat{\xi}\in W$ such that 
\begin{equation}\label{CrEq3}
\widehat{\beta}(x,y)=\langle x,y\rangle\widehat{\xi} \ \ \text{for all}\ x\in \widehat{V}_1\ \text{and}\ y\in V.
\end{equation}
Moreover, since $\beta_m\in\Sigma_{n,k}$ there exists an open subset $\widehat{\CU}_m$ of 
$\Sp^{k-1}$ such that $\widehat{\CU}_m\subseteq \Omega(\widehat{\beta}_m)$ 
and $\det\widehat{\beta}_m^\sharp(u)\neq0$ for all $u\in \widehat{\CU}_m$ and 
$m\in \mathbb{N}.$

Let $\{\widehat{u}_m\}$ be a sequence such that $\widehat{u}_m\in \widehat{\CU}_m$ for all 
$m\in\mathbb{N}$. We may assume that 
$\{\widehat{u}_m\}$ is convergent, by passing if necessary to a subsequence and set 
$\widehat{u}=\lim_{m\rightarrow \infty}\widehat{u}_m$. Since 
$\lim_{m\rightarrow \infty}\widehat{\beta}_m^\sharp(\widehat{u}_m)=\widehat{\beta}^\sharp(\widehat{u})$ 
and $\widehat{u}_m\in\widehat{\CU}_m$ it follows that $\mathrm{Index}\ \widehat{\beta}^\sharp(\widehat{u})<n-k$.
Then, from (\ref{CrEq3}) we get $\langle \widehat{\xi},\widehat{u}\rangle\geq 0$. 
We claim that $\langle\widehat{\xi},\widehat{u}\rangle=0$. Indeed, if 
$\langle\widehat{\xi},\widehat{u}\rangle>0$ then (\ref{CrEq3}) implies that 
$\widehat{\beta}^\sharp(\widehat{u})$ has at least $n-k$ positive eigenvalues and so, 
for $m$ large enough $\widehat{\beta}_m^\sharp(\widehat{u}_m)$ has at least 
$n-k$ positive eigenvalues. On account of the fact that $\det \widehat{\beta}_m^\sharp(u)\neq 0$ 
for all $u\in \widehat{\CU}_m$ and $m\in \mathbb{N}$, we have that 
$\widehat{\beta}_m^\sharp(\widehat{u}_m)$ has at most $k$ negative eigenvalues,
 which is a contradiction, since $\widehat{u}_m\in \widehat{\CU}_m$ for all $m\in \mathbb{N}$. 

Thus, we have proved that for any convergent sequence $\{\widehat{u}_m\}$ such that $\widehat{u}_m\in \widehat{\CU}_m$ for all $m\in\mb{N}$, we have $\langle\lim_{m\rightarrow \infty}\widehat{u}_m,\widehat{\xi}\rangle=0$.
Since $\widehat{\CU}_m$ is open, we may choose convergent sequences  
$\{\widehat{u}_m^{(1)}\},\{\widehat{u}_m^{(2)}\},\cdots,\{\widehat{u}_m^{(k)}\}$ in $\widehat{\CU}_m$ 
such that $\widehat{u}_m^{(1)}, \widehat{u}_m^{(2)},\cdots, \widehat{u}_m^{(k)}$ span $W$ for all $m\in \mathbb{N}.$ Then, from (\ref{CrEq3}) and the fact that $\langle\lim_{m\rightarrow \infty}\widehat{u}_m^{(a)},\widehat{\xi}\rangle=0$ for all $a\in\{1,2,\cdots,k\}$ we have
that the restriction of $\widehat{\beta}_m$ to $\widehat{V}_1\times \widehat{V}_1$ satisfies
$$\lim_{m\rightarrow\infty} \widehat{\beta}_m\vert_{\widehat{V}_1\times\widehat{V}_1}=0$$ and consequently
\be\label{xx3}
\lim_{m\rightarrow\infty}{\sf R}(\widehat{\beta}_m)\vert_{\widehat{V}_1\times \widehat{V}_1\times \widehat{V}_1\times \widehat{V}_1} 
=0 \ \ \text{and}\ \
\lim_{m\rightarrow\infty} {\sf L }(\widehat{\beta}_m)\vert_{\widehat{V}_1\times \widehat{V}_1} = -\frac{\sca(\widehat{\beta})}{2(n-1)(n-2)}\left.\langle \cdot,\cdot\rangle\right\vert_{\widehat{V}_1\times \widehat{V}_1}.
\ee
From the inequality 
$$
\Vert\mathsf{W}(\widehat{\beta}_m)\vert_{\widehat{V}_1\times\widehat{V}_1\times
\widehat{V}_1\times\widehat{V}_1}\Vert^2\leq\phi_\lambda(\widehat{\beta}_m),
$$
(\ref{xx3}) and the fact that $\lim_{m\rightarrow \infty}\phi_\lambda(\widehat{\beta}_m)=0$ we obtain $\sca(\widehat{\beta})=0$,
which contradicts (\ref{CrEq2}). Thus, the sequence $\{\beta_m\}$ is bounded and converges to some $\beta \in\mathrm{Sym}(V\times V,W)$, by taking a subsequence if necessary.

We claim that $\phi_\lambda(\beta)>0$. Arguing indirectly, we assume that $\phi_\lambda(\beta)=0$. Then, we have $\mathsf{W}(\beta)=0$ and 
\begin{equation}\label{CrEq4}
\Vert\beta\Vert^2\leq \lambda \vert\tr \  {\beta}\vert^2=\lambda(\sca(\beta)+\Vert\beta\Vert^2).
\end{equation}
According to Lemma \ref{JDMLe2}, there exist a vector subspace $V_1$ of $V$ with $\dim V_1\geq n-k$ and a vector $\xi\in W$ such that 
\begin{equation}\label{CrEq5}
\beta(x,y)=\langle x,y\rangle\xi \ \ \text{for all}\ x\in V_1\ \text{and}\ y\in V.
\end{equation}
We notice that $\beta\neq0$. Indeed, if $\beta=0$, then $\beta^\sharp(u)=0$ for all $u\in \Sp^{k-1}$. Since $\beta_m\in \Sigma_{n,k}$ for all $m\in\mb{N}$ there exists $\xi_m\in \Omega(\beta_m)$ such that 
\begin{equation}\label{CrEq6}
|\det\beta_m^\sharp(\xi_m)|\Vol(\Omega(\beta_m))=1\ \ \text{for all} \ m\in\mathbb{N}.
\end{equation}
We may assume that the sequence $\{\xi_m\}$ converges to some $\xi\in\Sp^{k-1}$, by passing again to a subsequence if necessary. Then $\lim_{m\rightarrow \infty}\beta_m^\sharp(\xi_m)=\beta^\sharp(\xi)=0,$ which contradicts (\ref{CrEq6}). Therefore $\beta\neq0$.
   
From the fact that $\beta_m\in \Sigma_{n,k}$ for all $m\in \mathbb{N}$, we deduce that there exists an open subset $\CU_m$ of $\Sp^{k-1}$ such that
$\CU_m\subseteq \Omega(\beta_m)$ and $\det\beta_m^\sharp(u)\neq0$ for all $u\in \CU_m$ and $m\in \mathbb{N}.$
Let $\{u_m\}$ be a sequence such that $u_m\in \CU_m$ for all $m\in\mathbb{N}$. We may assume that $u_m$ is convergent, by passing if necessary to a subsequence, and set $u=\lim_{m\rightarrow \infty}u_m$. Since $\lim_{m\rightarrow \infty}\beta_m^\sharp(u_m)=\beta^\sharp(u)$ and $u_m\in\CU_m$, it follows that $\mathrm{Index}\ \beta^\sharp(u)<n-k$. Then, from (\ref{CrEq5}) we get $\langle \xi,u\rangle\geq 0$. We claim that  $\langle\xi,u\rangle=0$. Indeed, if $\langle\xi,u\rangle>0$, then (\ref{CrEq5}) implies that $\beta^\sharp(u)$ has at least $n-k$ positive eigenvalues and so, for $m$ large enough, $\beta_m^\sharp(u_m)$ has at least $n-k$ positive eigenvalues. This and the fact that $\det \beta_m^\sharp(u) \neq 0$ for all $u\in \CU_m$ and $m\in \mathbb{N},$
shows that $\beta_m^\sharp(u_m)$ has at most $k$ negative eigenvalues, which is a contradiction, since $u_m\in \CU_m$ for all $m\in \mathbb{N}$. 

 Thus, we have proved that for any convergent sequence $\{u_m\}$ such that $u_m\in \CU_m$ for all $m\in\mb{N}$, we have $\langle\lim_{m\rightarrow \infty} u_m,\xi\rangle=0.$
Since $\CU_m$ is open, we may choose convergent sequences $\{u_m^{(1)}\},\{u_m^{(2)}\},\cdots,\{u_m^{(k)}\}$ in $\CU_m$ such that $u_m^{(1)}, u_m^{(2)},\cdots, u_m^{(k)}$ span $W$ for all $m\in \mathbb{N}.$ From (\ref{CrEq5}) and the fact that $\langle\lim_{m\rightarrow\infty} u_m^{(a)},\xi\rangle=0$ for all $a\in\{1,2,\cdots,k\}$, we obtain that the restriction of $\beta_m$ to $V_1\times V_1$ satisfies
$$\lim_{m\rightarrow\infty} \beta_m\vert_{V_1\times V_1}=0$$ and consequently
\be\label{xx4}
\lim_{m\rightarrow\infty}{\sf R}(\beta_m)\vert_{V_1\times V_1\times V_1\times V_1} 
=0 \ \ \text{and}\ \
\lim_{m\rightarrow\infty} {\sf L }(\beta_m)\vert_{V_1\times V_1} = -\frac{\sca(\beta)}{2(n-1)(n-2)}\left.\langle \cdot,\cdot\rangle\right\vert_{V_1\times V_1}.
\ee
From the inequality 
$$
\Vert\mathsf{W}(\beta_m)\vert_{V_1\times V_1\times V_1\times V_1}\Vert^2\leq\phi_\lambda(\beta_m),
$$
(\ref{xx4}) and the fact that $\lim_{m\rightarrow \infty}\phi_\lambda(\beta_m)=0$ we obtain $\sca(\beta)=0$,
which contradicts (\ref{CrEq4}). 
Thus, we have proved that $\phi_\lambda(\beta)>0$ and so $\phi_\lambda$ attains a positive minimum on $\Sigma_{n,k}$ which obviously depends
only on $n,k$ and $\lambda$ and is denoted by $\varepsilon_1(n,k,\lambda)$. 

Now, let $\beta\in \mathrm{Sym}(V\times V, W)$. Assume that $\psi(\beta)\neq 0$ and set $\gamma=\beta/(\psi(\beta))^{1/n}.$ Clearly 
$\gamma\in\Sigma_{n,k}$, and consequently $\phi_\lambda(\gamma)\geq \varepsilon_1(n,k,\lambda).$ Since $\phi_\lambda$ is 
homogeneous of degree $4$, the desired inequality is obviously fulfilled. In the case where $\psi(\beta)=0$, the inequality is trivial. \qed
\end{proof}

\begin{remark}\po\rm{
In the case where $\lambda\leq 1/n$, arguing as in the proof of Proposition \ref{CrProp2}, 
we have that there exist a positive constant $d(n,k)$ such that 
$$\big(\Vert\beta\Vert^2-\lambda\vert\tr \  {\beta}\vert^2\big)^2\geq 
\big(\Vert\beta\Vert^2-\frac{1}{n}\vert\tr \  {\beta}\vert^2\big)^2\geq 
d(n,k)\Big(\int_{\Sp^{k-1}}|\det\beta^\sharp(u)|\ d\Sp_u\Big)^{4/n}$$  
for any $\beta\in\mathrm{Sym}(V\times V,W)$.

However, if $\lambda\in(1/n,1)$ 
then the first term of the LHS of inequalities (\ref{CrIneq2}) 
and (\ref{CrIneq}) is essential and cannot be dropped. For instance,
let $n\geq 7, \ k=2$ and $\{\xi_1,\xi_2\}$ be an orthonormal 
basis of $W$. Consider $\beta\in\mathrm{Sym}(V\times V,W)$ defined by
$\beta(x,y)=\langle Ax,y\rangle \xi_1,$
where $A=\mathrm{diag}(a,a,-a,-a,\cdots,-a), \ a>0.$
For any $\lambda\in [n/(n-4)^2,1)$ we have that $\big(\Vert\beta\Vert^2-\lambda\vert\tr \  {\beta}\vert^2\big)_+=0$. 
Moreover, $\Phi(\beta)=\Sp^{1}\smallsetminus\{\pm\xi_2\}$ and this
shows that inequality (\ref{CrIneq2})  cannot hold by dropping the first term of the LHS.
 }
\end{remark}

\section{The proofs}

We recall some well known facts on the total curvature and how Morse 
theory provides restrictions on the Betti numbers.
Let $f\colon (M^n,g)\rightarrow \mb{R}^{n+k}$ be an isometric immersion of a 
compact, connected and oriented $n$-dimensional Riemannian manifold into 
the $(n+k)$-dimensional Euclidean space $\R^{n+k}$ equipped with the usual 
inner product $\langle\cdot,\cdot\rangle$. The normal bundle of $f$ is given by 
$$N_fM=\left\{(p,\xi)\in f^*(T\mb{R}^{n+k})\ :\ \xi\perp df_p(T_pM) \right\}$$ and 
the corresponding unit normal bundle is defined by 
$$UN_f=\left\{(p,\xi)\in N_fM\ :\ \vert\xi\vert=1\right\},$$
where $f^*(T\mb{R}^{n+k})$ is the induced bundle of $f$.

The {\it generalized Gauss map} $\nu\colon UN_f\rightarrow \mb{S}^{n+k-1}$ is 
defined by $\nu(p,\xi)=\xi$, where $\Sp^{n+k-1}$ is the unit $(n+k-1)$-dimensional 
sphere of $\R^{n+k}$. For each $u\in\mb{S}^{n+k-1}$, we consider the height function 
$h_u\colon M^n\rightarrow \mb{R}$ defined by $h_u(p)=\langle f(p),u\rangle,\ p\in M^n$. 
Since $h_u$ has a degenerate critical point if and only if $u$ is a critical point of the 
generalized Gauss map, by Sard's theorem there exists a subset $E\subset\Sp^{n+k-1}$ 
of measure zero such that $h_u$ is a Morse function for all $u\in\Sp^{n+k-1}\smallsetminus E$. 
For each $u\in \Sp^{n+k-1}\smallsetminus E$, we denote by $\mu_i(u)$ the number of critical 
points of $h_u$ of index $i$. We also set $\mu_{i}(u)=0$ for any $u\in E$. Following Kuiper 
\cite{Kuiper}, we define the {\it total curvature of index $i$ of $f$} by 
$$
\tau_i(f)=\frac{1}{\mathrm{Vol}(\mb{S}^{n+k-1})}\int_{\mb{S}^{n+k-1}}\mu_i(u)\ d\mb{S},
$$
where $d\Sp$ denotes the volume element of the sphere $\Sp^{n+k-1}$.

Let $\beta_i(M^n; \F)=\dim_\F H_i(M^n;\F)$ be the $i$-th Betti number of $M^n$ over an arbitrary coefficient field 
$\F$. From the weak Morse inequalities (cf. \cite{Milnor}) we have $\mu_i(u)\geq \beta_i(M^n;\F)$, 
for all $u\in \Sp^{n+k-1}\smallsetminus E$. By integrating over $\Sp^{n+k-1}$, we obtain 
\begin{equation}\label{TC}
\tau_i(f)\geq \beta_i(M^n;\F).  
\end{equation}

For each $(p,\xi)\in UN_f$, we denote by $A_{\xi}$ the shape operator of $f$ in the direction 
$\xi$ which is given by $$g (A_{\xi} X,Y)=\langle \alpha(X,Y),\xi\rangle,\ \ X,Y\in TM,$$ where 
 $\alpha$ is the second fundamental form of $f$ viewed as a section of the vector bundle 
 $\mathrm{Hom}(TM\times TM,N_f M)$. There is a natural 
volume element $d\Sigma$ on the unit normal bundle $UN_f$. In fact, if $dV$ is a $(k-1)$-form on 
$UN_f$ such that its restriction to a fiber of the unit normal bundle at $(p,\xi)$ is the volume
element of the unit $(k-1)$-sphere of the normal space of $f$ at $p$,
then $d\Sigma=dM\wedge dV$, where $dM$ is the volume element of $M^n$ with 
respect to the metric $g$. Furthermore, we have $$\nu^*(d\Sp)=G(p,\xi) d\Sigma,$$
where $G(p,\xi)=(-1)^n\det A_\xi$ is the Lipschitz-Killing curvature at $(p,\xi)\in UN_f$.

The total absolute curvature $\tau(f)$ of $f$ in the sense of Chern and Lashof is defined by 
$$
\tau(f)=\frac{1}{\Vol(\Sp^{n+k-1})}\int_{UN_f}\vert \nu^*(d\Sp)\vert=\frac{1}{\Vol(\Sp^{n+k-1})}
\int_{UN_f}\vert\det A_\xi\vert\ d\Sigma.
$$ 

We need the following result which is due to Chern and Lashof \cite{CL1,CL2}.
\begin{theorem}\po\label{ChLa}
Let $f\colon M^n\rightarrow \R^{n+k}$ be an isometric immersion of 
a compact, connected and oriented $n$-dimensional Riemannian manifold $M^n$ into 
$\R^{n+k}$. Then the total absolute curvature of $f$ satisfies the 
inequality 
$$\tau(f)\geq \sum_{i=0}^n\beta_{i}(M^n;\F).$$
\end{theorem}

Shiohama and Xu \cite[p. 381]{SX} proved that 
\begin{equation}\label{ShXu}
\int_{U^iN_f}|\mathrm{det}A_\xi |\ d\Sigma=\int_{\mb{S}^{n+k-1}}\mu_i(u)\ d\Sp,
\end{equation}
where $U^iN_f$, is the subset of the unit normal 
bundle of $f$ defined by
$$U^iN_f=\left\{(p,\xi)\in UN_f\ :\ \mathrm{Index}\ A_\xi=i\right\},\  \ 0\leq i\leq n.$$

The $(0,4)$-Riemann curvature tensor $R$ of $M^n$ is related to the 
second fundamental form of $f$ via the Gauss equation 
$$R(X,Y,Z,W)=\langle \alpha(X,Z),\alpha(Y,W)\rangle-
\langle\alpha(X,W),\alpha(Y,Z)\rangle,\ \ X,Y,Z,W\in TM.$$ 
In terms of the Kulkarni-Nomizu product, the Gauss equation is written equivalently as 
$$
R=\frac{1}{2}\alpha\varowedge\alpha.
$$
On the other hand, $R$ decomposes as
$$
R=\mathcal{W}+\Sh\varowedge g,
$$
where $\W$ is the {\it Weyl tensor} and
$$\Sh=\frac{1}{n-2}\Big(\Ric-\frac{\scal}{2(n-1)}g\Big).$$
is the Schouten tensor of $M^n$.

We are now able to present the proofs of our results.

\smallskip

\noindent\emph{Proof of Theorem \ref{MainTh2}:} 
Let $f\colon M^n\rightarrow \R^{n+k}$ be an isometric immersion with second fundamental 
form $\alpha$ and shape operator $A_\xi$ with respect to $\xi$, where $(p,\xi)\in UN_f$. 
Using the Gauss equation
and Proposition \ref{CrProp2} we have
 $$\Big(\Big\Vert R-\frac{\scal}{n(n-1)} R_1\Big\Vert^2+\big(S-\delta n^2H^2\big)_+^2\Big)^{n/4}(p)
 \geq (\varepsilon(n,k,\delta))^{n/4} \int_{\Lambda(\alpha(p))}\vert\det A_\xi \vert\ dV_\xi$$
 for all $p\in M^n$.
  Integrating over $M^n$ and using (\ref{ShXu}), we obtain 
\begin{equation}\label{32m32} 
\int_{M^n} \Big(\Big\Vert R-\frac{\scal}{n(n-1)}R_1\Big\Vert^2+\big(S-\delta n^2H^2\big)_+^{2}\Big)^{n/4} dM
\geq(\varepsilon(n,k,\delta))^{n/4} \Vol(\Sp^{n+k-1})\sum_{i=k}^{n-k} \tau_i(f). 
\end{equation}
Observe that 
$$\Big(\Big\Vert R-\frac{\scal}{n(n-1)}R_1\Big\Vert^2+(S-\delta n^2H^2)_+^{2}\Big)^{n/4}(p)\leq$$ 
\be\label{4965}
2^{(n-4)/4} \Big(\Big\Vert R-\frac{\scal}{n(n-1)}R_1\Big\Vert^{n/2}+\big(S-\delta n^2H^2\big)_+^{n/2}\Big)(p)
\ee
for all $p\in M^n$. Thus, from (\ref{32m32}) and (\ref{TC}) we obtain 
$$
\int_{M^n} \Big\Vert R-\frac{\scal}{n(n-1)}R_1\Big\Vert^{n/2} dM+ \int_{M^n} \big(S-\delta n^2H^2\big)_+^{n/2}\  dM
\geq c(n,\delta)\sum_{i=k}^{n-k} \tau_i(f)
$$
\begin{equation}\label{57667}
\geq c(n,\delta)\sum_{i=k}^{n-k} \beta_i(M;\F),
\end{equation}
 where 
$$c(n,\delta)=\mathrm{min}_{2\leq k\leq n/2}\left\{2\Big(\frac{\varepsilon(n,k,\delta)}{2}\Big)^{n/4}\Vol(\Sp^{n+k-1})\right\}.$$ 

Now, assume that
$$\int_{M^n}\Big\Vert R-\frac{\scal}{n(n-1)}R_1\Big\Vert^{n/2} dM+ 
\int_{M^n} \big(S-\delta n^2H^2\big)_+^{n/2}\  dM<c(n,\delta).$$ 
Then it follows directly from $(\ref{57667})$ that $\sum_{i=k}^{n-k} \tau_i(f)<1.$ Thus, there exists $u\in \Sp^{n-k-1}$ such that the height function $h_u\colon M^n\rightarrow \R$ is a Morse function whose number of critical points of index $i$ satisfies $\mu_i(u)=0$ for any $k\leq i\leq n-k$. The fundamental theorem of Morse theory (cf. \cite[Theorem 3.5]{Milnor} or \cite[Theorem 4.10]{CE}) then implies that $M^n$ has the 
homotopy type of a CW-complex with no cells of dimension $i$ for $k\leq i\leq n-k$. 

Now, if $k=2$, there will be no $2$-cells and thus by the cellular approximation theorem we conclude that the inclusion 
of the $1$-skeleton $\mathrm{X}^{(1)}\hookrightarrow M^n$ induces isomorphism between 
the fundamental groups. Therefore, the fundamental group $\pi_1(M^n)$ is a free group on 
$\beta_1(M^n;\Z)$ elements and $H_1(M^n;\Z)$ is a free abelian group on $\beta_1(M^n;\Z)$ generators. 
In particular, if $\pi_1(M^n)$ is finite, then $\pi_1(M^n)=0$ and hence $H_1(M^n;\Z)=0$. 
From Poincar\'e duality and the universal coefficient theorem it follows that $H_{n-1}(M^n;\Z)=0$. 
Thus, $M^n$ is a simply connected homology sphere and hence a homotopy sphere. By the generalized 
Poincar\'e conjecture (Smale $n\geq 5$, Freedman $n=4$) we deduce that $M^n$ is homeomorphic to $\Sp^n$.

 If the scalar curvature is everywhere non-positive, then from the Gauss equation we obtain 
 $S\geq \delta n^2 H^2$. Using Proposition \ref{CrProp2} and the Gauss equation 
 we have
$$\Big(\Big\Vert R-\frac{\scal}{n(n-1)}R_1\Big\Vert^2+\big(S-\delta n^2H^2\big)^2\Big)^{n/4}(p)
\geq (\varepsilon(n,k,\delta))^{n/4} \int_{\Sp_p^{k-1}}\vert\det A_\xi \vert\ dV_\xi$$
 for all $p\in M^n$.  Integrating over $M^n$, we obtain 
$$
\int_{M^n} \Big(\Big\Vert R-\frac{\scal}{n(n-1)}R_1\Big\Vert^2+\big(S-\delta n^2H^2\big)^{2}\Big)^{n/4} dM
\geq (\varepsilon(n,k,\delta))^{n/4} \int_{UN_f}|\det A_\xi| d\Sigma.$$
Bearing in mind the definition of the total absolute curvature $\tau(f)$ of $f$ and (\ref{4965}), we have 
$$
\int_{M^n}\Big\Vert R-\frac{\scal}{n(n-1)}R_1\Big\Vert^{n/2}\ dM+\int_{M^n}\big(S-\delta n^2H^2\big)^{n/2} dM
$$
\begin{equation}\label{0987}
 \geq 2(\varepsilon(n,k,\delta)/2)^{n/4}\Vol(\Sp^{n+k-1})\tau(f)\geq c(n,\delta)\tau(f).
\end{equation}
The desired inequality follows from Theorem \ref{ChLa}.

 If the scalar curvature is everywhere non-positive and 
 $$\int_{M^n}\Big\Vert R-\frac{\scal}{n(n-1)}R_1 \Big\Vert^{n/2}\  dM
 +\int_{M^n} \big(S-\delta n^2H^2\big)^{n/2}\ dM < 3c(n,\delta)$$ then from (\ref{0987}) 
 we obtain $\tau(f)<3$. This implies that there exists a height function which is a Morse 
 function with exactly two critical points. Reeb's theorem then implies that $M^n$ is homeomorphic to $\Sp^n$.  \qed

\smallskip

\noindent\emph{Proof of Corollary \ref{pinched3}:} 
Our assumptions and Theorem \ref{MainTh2} imply that $\beta_i(M^n;\F)=0,\ k\leq i\leq n-k$. Hence, 
$f$ is $\delta$-pinched and the rest of the proof follows from Theorem \ref{MainTh2}. Moreover, 
if the mean curvature is everywhere positive and $f$ is $1/(n-1)$-pinched, then 
a result due to Andrews and Baker \cite{Andrews} implies 
that $M^n$ is diffeomorphic to $\Sp^n$.
\qed

\smallskip

\noindent\emph{Proof of Corollary \ref{Cor1}:} 
Our assumptions and Theorem \ref{MainTh2} imply
$R=\big(\scal/n(n-1)\big)R_1.$ It follows from Shur's lemma that $M^n$ is a space form. According to a result due to
 Chern, Otsuki and Kuiper (cf. \cite[Corollary 4.8]{KN}) the sectional curvature must be positive. 
 Appealing to Moore \cite[Proposition 4]{JDM1}, $M^n$ is isometric to a constant curvature sphere.\qed

\smallskip

\noindent\emph{Proof of Corollary \ref{ThmMin2}:} 
We consider the immersion $\widetilde{f}=i\circ f$, where $i\colon \Sp^{n+k-1}\hookrightarrow\R^{n+k}$ is the totally umbilic inclusion and the proof follows 
directly from Theorem \ref{MainTh2}. \qed

\smallskip

\noindent\emph{Proof of Theorem \ref{MainTh}:} 
Let $f\colon (M^n,g)\rightarrow \mathbb{R}^{n+k}$ be a conformal immersion with second
fundamental form $\alpha$ and shape operator $A_\xi$ with respect to $\xi$, where $(p,\xi)\in UN_f$. Using the Gauss 
equation, it follows that the Weyl tensor $\mathcal{W}_{\tilde{g}}$ with respect to the induced metric $\tilde{g}$ of $f$ is given by 
$\mathcal{W}_{\tilde{g}}(p)=\mathsf{W}(\alpha(p)).$
From Proposition \ref{CrProp}, we have 
 $$\Big(\Vert\W_{\tilde{g}}\Vert^{2}+\big(S-\delta n^2H^2\big)_+^{2}\Big)^{n/4}(p)
\geq (\varepsilon_1(n,k,\delta))^{n/4}\int_{\Omega(\alpha(p))}|\det A_\xi| dV_\xi$$ for all $p\in M^n$. 
By integrating over $M^n$ and using (\ref{ShXu}), we obtain 
\begin{equation}\label{MTh1}
\int_{M^n} \Big(\Vert\W_{\tilde{g}}\Vert^{2}+\big(S-\delta n^2H^2\big)_+^{2}\Big)^{n/4} \ dM 
\geq (\varepsilon_1(n,k,\delta))^{n/4}\Vol(\Sp^{n+k-1})\sum_{i=k+1}^{n-k-1} \tau_i(f).
\end{equation}
Observe that 
$$
\Big(\Vert\W_{\tilde{g}}\Vert^{2}+\big(S-\delta n^2H^2\big)_+^{2}\Big)^{n/4}(p)
\leq 2^{(n-4)/4} \left( \Vert\W_{\tilde{g}}\Vert^{n/2}+\big(S-\delta n^2H^2\big)_+^{n/2}\right)(p)
$$
for all $p\in M^n$. 
Thus, from (\ref{MTh1}), (\ref{TC}) and the fact that the $L^{n/2}$-norm of the Weyl tensor 
is conformally invariant,  we have that
\begin{equation}\label{MTh3}
\int_{M^n}\Vert \W \Vert^{n/2} dM+ \int_{M^n} \big(S-\delta n^2H^2\big)_+^{n/2}  dM
\geq c_1(n,\delta)\sum_{i=k+1}^{n-k-1}\tau_i(f)\geq c_1(n,\delta)\sum_{i=k+1}^{n-k-1} \beta_i(M^n;\F),
\end{equation}
where 
$$c_1(n,\delta)=\mathrm{min}_{2\leq k\leq [(n-2)/2]}\left\{2\Big(\frac{\varepsilon_1(n,k,\delta)}{2}\Big)^{n/4}\Vol(\Sp^{n+k-1})\right\}.$$

Now, assume that
$$\int_{M^n}\Vert \mathcal{W} \Vert^{n/2}\ dM+ \int_{M^n}  \big(S-\delta n^2H^2\big)_+^{n/2}\  dM<c_1(n,\delta).$$ Then it follows from (\ref{MTh3}) that $\sum_{i=k+1}^{n-k-1} \tau_i(f)<1.$ Thus, there exists $u\in \Sp^{n+k-1}$ such that the height function $h_u\colon M^n\rightarrow \R$ is a Morse function whose number of critical points of index $i$ satisfies $\mu_i(u)=0$ for any $k<i<n-k$. The fundamental theorem of Morse theory then implies that $M^n$ has the homotopy type of a CW-complex with no cells of dimension $i$ for $k<i<n-k$. \qed

{\renewcommand{\baselinestretch}{1}
\hspace*{-20ex}\begin{tabbing} \indent\={\small\textsc{Department of Mathematics, University of Ioannina, 45110 Ioannina, Greece}} \\
\> {\small {\sc E-mail adrresses:} \texttt{tvlachos@uoi.gr, chonti@cc.uoi.gr}}
\end{tabbing}}

\end{document}